\documentclass[11pt]{scrartcl}
\usepackage{amsfonts,amssymb,amsthm,amsmath,booktabs}
\usepackage{color,graphicx}
\usepackage[latin1]{inputenc}

\ifx\pdftexversion\undefined
\usepackage{nohyperref}
\else
\usepackage[colorlinks=true,linkcolor={blue},citecolor={blue},urlcolor={blue},
  pdfauthor={Thomas Apel, Serge Nicaise, Johannes Pfefferer},
  pdfstartview={Fit}]{hyperref}
\fi

\newcommand{\abssec}[1]{\noindent\normalsize {\bfseries #1\quad }\ignorespaces}
\renewenvironment{abstract}{\abssec{Abstract}}{\par\vspace{.1in}}
\newenvironment{keywords}{\abssec{Key Words}}{\par\vspace{.1in}}
\newenvironment{AMSMOS}{\abssec{AMS subject
  classification}}{\par\vspace{.1in}}

\theoremstyle{plain}
\newtheorem{theorem}{Theorem}
\newtheorem{corollary}[theorem]{Corollary}

\newtheorem{lemma}[theorem]{Lemma}

\theoremstyle{definition}

\numberwithin{equation}{section}

\def\R{\mathbb{R}}

\def\S{\mbox{$\xi(r)\,r^\lambda\sin(\lambda\theta)$}}
\def\Sd{p_s}
\def\Fd{\phi_s}

\newcommand{\V}{H^1_\Delta(\Omega)\cap H^1_0(\Omega)}

\newcommand{\Span}{\text{Span}}

\definecolor{darkgreen}{rgb}{0.0, 0.5, 0.3}

\begin{document}

\title{\LARGE A dual singular complement method \\ for the
  numerical solution of the Poisson equation \\ with $L^2$ boundary
  data in non-convex domains\thanks{The work was partially supported
    by Deutsche Forschungsgemeinschaft, IGDK 1754.}}

\author{Thomas Apel\thanks{\texttt{thomas.apel@unibw.de},
    Universit\"at der Bundeswehr M\"unchen, Institut f\"ur Mathematik
    und Bauinformatik, D-85579 Neubiberg, Germany} \and Serge
  Nicaise\thanks{\texttt{snicaise@univ-valenciennes.fr}, LAMAV,
    Institut des Sciences et Techniques de Valenciennes, Universit\'e
    de Valenciennes et du Hainaut Cambr\'esis, B.P. 311, 59313
    Valenciennes Cedex, France } \and Johannes
  Pfefferer\thanks{\texttt{johannes.pfefferer@unibw.de}, Universit\"at
    der Bundeswehr M\"unchen, Institut f\"ur Mathematik und
    Bauinformatik, D-85579 Neubiberg, Germany}}
\maketitle

\begin{abstract}
    The very weak solution of the Poisson equation with $L^2$
    boundary data is defined by the method of transposition. The
    finite element solution with regularized boundary data converges
    with order $1/2$ in convex domains but
  has a reduced convergence order in
  non-convex domains. As a remedy, a dual variant of the singular
  complement method is proposed. The error order of the convex case is
  retained. Numerical experiments confirm the theoretical results.
\end{abstract}


\begin{AMSMOS}
  65N30; 65N15
\end{AMSMOS}

\section{Introduction}

In this paper we consider the  boundary value problem
\begin{align} \label{eq:bvp}
  -\Delta y &= f \quad\text{in }\Omega, & 
  y &= u \quad\text{on }\Gamma=\partial\Omega,
\end{align}
with right hand side $f\in H^{-1}(\Omega)$ and boundary data $u\in L^2(\Gamma)$. We assume $\Omega\subset\R^2$
to be a bounded polygonal domain with boundary $\Gamma$.  Such
problems arise in optimal control when the Dirichlet boundary control
is considered in $L^2(\Gamma)$ only, see for example the papers by
Deckelnick, G\"unther, and Hinze, \cite{DeckelnickGuentherHinze2009},
French and King, \cite{FrenchKing1991}, and May, Rannacher, and
Vexler, \cite{MayRannacherVexler2008}.

For boundary data $u\in L^2(\Gamma)$ we cannot expect a weak solution
$y\in H^1(\Omega)$. Therefore we define a very weak solution by the
method of transposition which goes back at least to Lions and Magenes
\cite{LionsMagenes1968}: Find
\begin{align}\label{eq:veryweak2}
  y\in L^2(\Omega):\quad (y, \Delta v)_\Omega =
  (u,\partial_n v)_\Gamma - (f,v)_\Omega \quad\forall
  v\in V
\end{align}
with $(w,v)_G:=\int_G wv$ denoting the $L^2(G)$ scalar product
  or an appropriate duality product.
In our previous paper \cite{ApelNicaisePfefferer2014a} we  showed
that the appropriate space $V$ for the test functions is
\begin{align}\label{eq:H1Delta}
  V:=H^1_\Delta(\Omega)\cap H^1_0(\Omega) \quad\text{with}\quad
  H^1_\Delta(\Omega):=\{v\in H^1(\Omega): \Delta v\in L^2(\Omega)\}.
\end{align}
In particular it ensures $\partial_n v\in L^2(\Gamma)$ for $v\in V$
such that the formulation \eqref{eq:veryweak2} is well defined. We
proved the existence of a unique solution $y\in L^2(\Omega)$ for $u\in
L^2(\Gamma)$ and $f\in H^{-1}(\Omega)$, and that the solution is even in
$H^{1/2}(\Omega)$. The method of transposition is used in different
variants also in \cite{FrenchKing1991,Berggren2004,CasasRaymond2006,%
  CasasMateosRaymond2009,DeckelnickGuentherHinze2009,MayRannacherVexler2008}.

Consider now the discretization of the boundary value problem.  Let
$\mathcal{T}_h$ be a family of quasi-uniform, conforming finite
element meshes, and introduce the finite element spaces
\begin{align*}
  Y_h = \{v_h\in H^1(\Omega): v_h|_T\in\mathcal{P}_1\ \forall
  T\in\mathcal{T}_h\}, \quad Y_{0h} = Y_h\cap H^1_0(\Omega),\quad
  Y_h^\partial = Y_h|_{\partial\Omega}.
\end{align*}
Since the boundary datum $u$ is in general not contained in
$Y_h^\partial$ we have to approximate it by $L^2(\Gamma)$-projection
or by quasi-interpolation. We showed in
\cite{ApelNicaisePfefferer2014a} that we can construct in this way a
function $u^h$ with
\begin{align*}
  \|u-u^h\|_{H^{-1/2}(\Gamma)}\le Ch^{1/2}\|u\|_{L^2(\Gamma)}.
\end{align*}
As a side effect, the boundary datum is regularized since $u^h\in
H^{1/2}(\Gamma)$. Hence we can consider a regularized (weak) solution
 in $Y_*^h:=\{v\in H^1(\Omega): v|_\Gamma=u^h\}$,  
\begin{align}\label{eq:regsol}
  y^h\in Y_*^h: \quad (\nabla y^h,\nabla v)_\Omega =
  (f,v)_\Omega \quad\forall v\in H^1_0(\Omega).
\end{align}
The finite element solution $y_h$ is now searched in $Y_{*h}:=
Y_*^h\cap Y_h$ and is defined in the classical way: find
\begin{align}\label{eq:serge20/06:6}
  y_h\in Y_{*h}:\quad (\nabla y_h,\nabla v_h)_\Omega =
  (f,v_h)_\Omega\quad\forall v_h\in Y_{0h}.
\end{align}
The same discretization was derived previously by Berggren
\cite{Berggren2004} from a different point of view.
In \cite{ApelNicaisePfefferer2014a} we showed that the discretization
error estimate
\begin{align*}
  \|y-y_h\|_{L^2(\Omega)}\le Ch^s
  \left(h^{1/2}\|f\|_{H^{-1}(\Omega)}+\|u\|_{L^2(\Gamma)}\right)
\end{align*}
holds for $s=1/2$ if the domain is convex; this is a slight
improvement of the result of Berggren.

Let us now consider non-convex domains. Although the very weak
solution $y$ is also in $H^{1/2}(\Omega)$ the convergence order is
reduced; the finite element method does not lead to the best
approximation in $L^2(\Omega)$. In order to describe the result we
assume for simplicity that $\Omega$ has only one corner with interior
angle $\omega\in(\pi,2\pi)$. 
We proved in \cite{ApelNicaisePfefferer2014a} the convergence order
$s\in(0,\lambda-\frac12)$, where $\lambda:=\frac\pi\omega$, and showed
by numerical experiments that the order of almost $\lambda-\frac12$ is
sharp.

In this paper, we modify the discrete solution $y_h$ from
\eqref{eq:serge20/06:6} in order to retain the convergence order
$s=\frac12$. In particular, we suggest to compute a function \[ z_h\in
Y_h\oplus\Span\{r^{-\lambda}\sin(\lambda\theta)\},\] where $r,\theta$ are
polar coordinates at the concave corner, such that the error estimate
\begin{align*}
  \|y-z_h\|_{L^2(\Omega)}\le Ch^{1/2}
  \left(h^{1/2}\|f\|_{H^{-1}(\Omega)}+\|u\|_{L^2(\Gamma)}\right)
\end{align*}
can be shown. This method is a dual variant of the singular complement
method introduced by Ciarlet and He \cite{ciarletjr:03}. Numerical
experiments confirm the theoretical results.

\section{Analytical background and regularization}

As in the introduction, let $\Omega$ be a domain with exactly one concave
corner, and denote this interior angle by $\omega\in(\pi,2\pi)$.  This
corner is located at the origin of the coordinate system, and one
boundary edge is contained in the positive $x_1$-axis. It is well
known that the weak solution of the boundary value problem 
\begin{align}\label{eq:thomas+}
  -\Delta v &= g \quad\text{in }\Omega, & 
  v &= 0 \quad\text{on }\Gamma=\partial\Omega,
\end{align}
with $g\in L^2(\Omega)$ is not contained in $H^2(\Omega)$ but in 
\[
  \V=\left(H^2(\Omega)\cap H^1_0(\Omega)\right)\oplus \Span \{\S\},
\]
$\xi$ being a cut-off function, see for example the monograph of
Grisvard \cite{grisvard:92b}.  This means that
\[
  R:=\{\Delta v: v\in H^2(\Omega)\cap H^1_0(\Omega)\},
\]
is a closed subspace of $L^2(\Omega)$. It is shown in \cite[Sect.
2.3]{grisvard:92b} that
\begin{align}\label{eq:serge20/06:1}
  L^2(\Omega)=R \overset{\perp}{\oplus} \Span \{\Sd\},
\end{align}
with the \emph{dual singular function}
\begin{align}\label{def:ps}
  p_s=r^{-\lambda}\sin(\lambda\theta)+\tilde p_s
\end{align}
where $\tilde p_s\in H^1(\Omega)$ is chosen such that the
decomposition \eqref{eq:serge20/06:1} is orthogonal for the
$L^2(\Omega)$ inner product. Therefore, the dual singular function
$p_s$ is a solution of
\begin{align}\label{eq:thomas*}
  w\in L^2(\Omega):\quad (\Delta v,w)=0 
  \quad\forall v\in H^2(\Omega)\cap H^1_0(\Omega),
\end{align}
which proves the non-uniqueness of the solution of \eqref{eq:thomas*}.
This is the dual property to the non-existence of a solution of
\eqref{eq:thomas+} in $H^2(\Omega)\cap H^1_0(\Omega)$, see
\cite[Introduction]{grisvard:92b}.

Due to \eqref{eq:serge20/06:1} we can split any $L^2(\Omega)$-function
into $L^2(\Omega)$-orthogonal parts. To this end denote by $\Pi_R$ and
$\Pi_{\Sd}$ the orthogonal projections on $R$ and on $\Span\{\Sd\}$,
respectively, i.e., for $g\in L^2(\Omega)$, it is
$g=\Pi_Rg+\Pi_{\Sd}g$ where
\begin{align*}
  \Pi_{\Sd} g&=\alpha(g)\, \Sd \quad\text{with}\quad
  \alpha(g) =\frac{(g,\Sd)_\Omega}{\|\Sd\|_{L^2(\Omega)}^2}, \\
  \Pi_R g&=g-\Pi_{\Sd} g.
\end{align*}
Since $\Sd\in L^2(\Omega)$  there exists 
\begin{align}\label{def:phis}
\Fd\in \V:\quad -\Delta \Fd= \Sd,
\end{align}
see also Section \ref{sec:scm} for more details on $\Fd$.
For the moment we assume that $\Sd$ and $\Fd$ are explicitly known;
hence the decomposition $g=\Pi_Rg+\alpha(g)\, \Sd$ can be computed
once $g$ is given. Computable approximations of $\Sd$ and $\Fd$ are
discussed in Section \ref{sec:scm}.

Now we come back to problem \eqref{eq:veryweak2} and decompose its
solution $y$ in the form
\begin{align}\label{eq:serge20/06:2}
y=\Pi_Ry+\alpha(y)\, \Sd.
\end{align}
From the decomposition \eqref{eq:serge20/06:1} we
see that problem \eqref{eq:veryweak2} is equivalent to
\begin{align*}
   (y, \Sd)_\Omega &=
  -(u,\partial_n \Fd)_\Gamma + (f,\Fd)_\Omega,
\\ 
 (y, \Delta v)_\Omega &=
  (u,\partial_n v)_\Gamma - (f,v)_\Omega \quad\forall
  v\in H^2(\Omega)\cap H^1_0(\Omega)
\end{align*}
and with the orthogonal splitting \eqref{eq:serge20/06:2} to
\begin{align*}
   \alpha(y)\,(\Sd, \Sd)_\Omega &=
  -(u,\partial_n \Fd)_\Gamma + (f,\Fd)_\Omega,
\\ 
 (\Pi_Ry, \Delta v)_\Omega &=
  (u,\partial_n v)_\Gamma - (f,v)_\Omega \quad\forall
  v\in H^2(\Omega)\cap H^1_0(\Omega).
\end{align*}
The first equation directly yields $\alpha(y)$, namely
\begin{align}\label{eq:serge20/06:7}
  \alpha(y) =\frac{-(u,\partial_n \Fd)_\Gamma+(f,\Fd)_\Omega}%
  {(\Sd, \Sd)_\Omega},
\end{align}
hence the projection of $y$ on $\Sd$ is known. It remains to find an
approximation of $\Pi_Ry$.

At this point we recall the regularization approach from
\cite{ApelNicaisePfefferer2014a} which we summarized already in the
introduction. Let $u^h\in H^{1/2}(\Gamma)$ be a regularized boundary
datum such that we can define the regularized (weak) solution in
$Y_*^h:=\{v\in H^1(\Omega): v|_\Gamma=u^h\}$,
\begin{align}\label{eq:regsol:repeat}
  y^h\in Y_*^h: \quad (\nabla y^h,\nabla v)_\Omega =
  (f,v)_\Omega \quad\forall v\in H^1_0(\Omega).
\end{align}
In \cite{ApelNicaisePfefferer2014a} we showed that the regularization
error can be estimated by
\begin{align*}
  \|y-y^h\|_{L^2(\Omega)}\le c \|u-u^h\|_{H^{-s}(\Gamma)}
\end{align*}
where $s=\frac12$ if $\Omega$ is convex and
$s\in[0,\lambda-\frac12)$ if $\Omega$ is non-convex, that means
the regularization error is in general bigger in the non-convex case.
With the next lemma we show that $\Pi_R(y-y^h)$ is not affected by
non-convex corners.

\begin{lemma}\label{lem:regularizationerrornonconvex}
  If the domain $\Omega$ is non-convex, the estimate
  \begin{align*}
    \|\Pi_R(y- y^h)\|_{L^2(\Omega)}\le C\|u-u^h\|_{H^{-1/2}(\Gamma)}
  \end{align*}
  holds.
\end{lemma}

\begin{proof}
  Recall $V=H^1_\Delta(\Omega)\cap H^1_0(\Omega)$ from \eqref{eq:H1Delta}.
  From \eqref{eq:regsol:repeat} and the Green formula, we have for any
  $v\in V$
  \begin{align*}
    (f,v)_\Omega = (\nabla y^h,\nabla v)_\Omega =
    -( y^h,\Delta v)_\Omega + (y^h, \partial_n v)_\Gamma .
  \end{align*}
  Note that $v\in V$ is sufficient, see  \cite[Lemma 3.4]{Costabel1988}.
  Subtracting this expression from the very weak formulation
  \eqref{eq:veryweak2}, we get
  \[
    (y-y^h,\Delta v)_\Omega=(u-u^h, \partial_n v)_\Gamma 
    \quad\forall v\in V.
  \]
  Restricting this identity to $v\in H^2(\Omega)\cap H^1_0(\Omega)$, we have
  \begin{align}\label{eq:serge20/06:5}
    ( \Pi_R(y-y^h),\Delta v)_\Omega=(u-u^h, \partial_n
    v)_\Gamma \quad \forall v\in H^2(\Omega)\cap H^1_0(\Omega).
  \end{align}
  Now for any $z\in R$, we let $v_z\in H^2(\Omega)\cap H^1_0(\Omega)$
  be the unique solution of
  \begin{align}\label{eq:sergepbDir}
    \Delta v_z=z,
  \end{align}
  that satisfies
  \begin{align}\label{eq:serge2}
    \|\partial_n v_z\|_{H^{1/2}(\Gamma)} \le 
    c\|v_z\|_{H^2(\Omega)}\le c \|z\|_{L^2(\Omega)}.
  \end{align}
  Since for any $g\in L^2(\Omega)$ the equality
  \[
    (\Pi_R(y-y^h),g)_\Omega=
    (\Pi_R(y-y^h),\Pi_Rg)_\Omega= 
    (y-y^{h},\Pi_Rg)_\Omega
  \]
  holds we get with \eqref{eq:serge20/06:5}--\eqref{eq:serge2}
  \begin{align*}
    \|\Pi_R(y-y^h)\|_{L^2(\Omega)}&= \sup_{z\in R, z\ne 0}
    \frac{(y-y^{h}, z)_\Omega}{\|z\|_{L^2(\Omega)}} =
    \sup_{z\in R, z\ne 0}
    \frac{(u-u^h,\partial_n v_z)_\Gamma}{\|z\|_{L^2(\Omega)}} \\ &\leq
    \|u-u^h\|_{H^{-1/2}(\Gamma)} \sup_{z\in R, z\ne 0}
    \frac{\|\partial_n v_z\|_{H^{1/2}(\Gamma)}}{\|z\|_{L^2(\Omega)}} \le c
    \|u-u^h\|_{H^{-1/2}(\Gamma)}
  \end{align*}
  which is the estimate to be proved.
\end{proof}

\section{\label{sec:3}Discretization by standard finite elements}

Recall from the introduction the finite element spaces
\begin{align*}
  Y_h = \{v_h\in H^1(\Omega): v_h|_T\in\mathcal{P}_1\ \forall
  T\in\mathcal{T}_h\}, \quad Y_{0h} = Y_h\cap H^1_0(\Omega),\quad
  Y_h^\partial = Y_h|_{\partial\Omega},
\end{align*}
defined on a family $\mathcal{T}_h$ of quasi-uniform, conforming
finite element meshes. Assume that the regularized boundary datum $u^h$ is
contained in $Y_h^\partial$ such that the estimates
\begin{align}\label{eq:uhstable}
  \|u^h\|_{L^2(\Gamma)}\le c\|u\|_{L^2(\Gamma)},\\ \label{eq:carstensen}
  \|u-u^h\|_{H^{-1/2}(\Gamma)}\le Ch^{1/2}\|u\|_{L^2(\Gamma)},
\end{align}
hold. It is proved in \cite{ApelNicaisePfefferer2014a} that this can
be accomplished by using the $L^2(\Gamma)$-projection or by
quasi-interpolation. A consequence of Lemma
\ref{lem:regularizationerrornonconvex} is the estimate
\begin{align}\label{est:regularizationerrornonconvex}
  \|\Pi_R(y- y^h)\|_{L^2(\Omega)}\le Ch^{1/2}\|u\|_{L^2(\Gamma)}
\end{align}
in the case of a non-convex domain $\Omega$. (In the case of a convex
domain the operator $\Pi_R$ is the identity, and the corresponding
error estimates were already proven in
\cite{ApelNicaisePfefferer2014a}.)

As already done in the introduction, define further the finite element solution $y_h\in Y_{*h}:= Y_*^h\cap
Y_h$ via
\begin{align}\label{eq:serge20/06:6repeat}
  y_h\in Y_{*h}:\quad (\nabla y_h,\nabla v_h)_\Omega =
  (f,v_h)_\Omega\quad\forall v_h\in Y_{0h}.
\end{align}
We proved in \cite{ApelNicaisePfefferer2014a} that
\begin{align}
  \|y-y_h\|_{L^2(\Omega)}\le Ch^s
  \left(h^{1/2}\|f\|_{H^{-1}(\Omega)}+\|u\|_{L^2(\Gamma)}\right)\label{jonny:fe_error}
\end{align}
holds for $s=\frac12$ if the domain is convex but only
$s\in(0,\lambda-\frac12)$ in the non-convex case.  In the next lemma
we show that $\Pi_R(y-y_h)$ is not affected by the non-convex corners.

\begin{lemma}\label{lem:3.1}
  For non-convex domains $\Omega$ the discretization error estimate
  \begin{align*}
    \|\Pi_R(y-y_h)\|_{L^2(\Omega)}\le C  h^{1/2} 
    \left(h^{1/2}\|f\|_{H^{-1}(\Omega)}+ \|u\|_{L^2(\Gamma)}\right)
  \end{align*}
  holds. 
\end{lemma}

\begin{proof}
  By the triangle inequality we have
  \begin{align} \label{thomas:triangleinequality0207}
    \|\Pi_R(y-y_h)\|_{L^2(\Omega)} &\le 
    \|\Pi_R(y-y^h)\|_{L^2(\Omega)} + \|\Pi_R(y^h-y_h)\|_{L^2(\Omega)}. 
  \end{align}
  The first term is estimated in
  \eqref{est:regularizationerrornonconvex}.  For the second term we
  first notice that $y^h-y_h\in H^1_0(\Omega)$ satisfies the Galerkin
  orthogonality
  \begin{align}\label{eq:y^h-y_hGalerkin}
    (\nabla (y^h-y_h),\nabla v_h)_\Omega =0\quad\forall v_h\in
    Y_{0h},
  \end{align}
  see \eqref{eq:regsol} and \eqref{eq:serge20/06:6}.
  With that, we estimate $\|\Pi_R(y^h-y_h)\|_{L^2(\Omega)}$ by a
  similar arguments as $\|\Pi_R(y-y^h)\|_{L^2(\Omega)}$ in the proof
  of Lemma \ref{lem:regularizationerrornonconvex}. Recall from
  \eqref{eq:sergepbDir} and \eqref{eq:serge2} that $v_z\in
  H^2(\Omega)\cap H^1_0(\Omega)$ is the weak solution of $\Delta v_z=z\in
  R$. It can be approximated by the Lagrange interpolant $I_hv_z$
  satisfying
  \[
    \|\nabla(v_z-I_h v_z)\|_{L^2(\Omega)}\le ch \|v_z\|_{H^2(\Omega)} 
    \le ch \|z\|_{L^2(\Omega)}.
  \]
  We get 
  \begin{align}
    \|\Pi_R(y^h-y_h)\|_{L^2(\Omega)}&=\sup_{z\in R, z\ne 0}
    \frac{(y^h-y_h, z)_\Omega}{\|z\|_{L^2(\Omega)}} =
    \sup_{z\in R, z\ne 0} \frac{(\nabla(y^h-y_h), 
    \nabla v_z)_\Omega}{\|z\|_{L^2(\Omega)}} \nonumber \\ &=
    \sup_{z\in R, z\ne 0} \frac{(\nabla(y^h-y_h), 
    \nabla (v_z-I_h v_z))_\Omega}{\|z\|_{L^2(\Omega)}} \nonumber  \\&\le 
    ch \|\nabla(y^h-y_h)\|_{L^2(\Omega)}.
    \label{eq:thomas0207}
  \end{align} 

  In order to bound $\|\nabla(y^h-y_h)\|_{L^2(\Omega)}$ by the data we
  consider a lifting $B_hu^h\in Y_{*h}$ defined by the nodal values as
  follows: 
  \begin{align}\label{def:Rh}
    (B_hu^h)(x)&= \begin{cases} u^h(x), &\text{for all nodes } x\in\Gamma,\\
    0 &\text{for all nodes } x\in\Omega.\end{cases}
  \end{align}
  The homogenized solution $y_0^h=y^h-B_hu^h\in H^1_0(\Omega)$ satisfies
  \begin{align*}
    (\nabla y_0^h,\nabla v)_\Omega = (f,v)_\Omega - 
    (\nabla (B_hu^h),\nabla v)_\Omega \quad\forall v\in H^1_0(\Omega).
  \end{align*}
  By taking $v=y_0^h$ we see that
  \begin{align*}
    \|\nabla y_0^h\|^2_{L^2(\Omega)} \le \|f\|_{H^{-1}(\Omega)}\|y_0^h\|_{H^1(\Omega)}
    +\|\nabla (B_hu^h)\|_{L^2(\Omega)} \|\nabla y_0^h\|_{L^2(\Omega)}.
  \end{align*}
  Using the Poincar\'e inequality we obtain
  \begin{align}\label{eq:serge7}
    \|\nabla y_0^h\|_{L^2(\Omega)} \le c\|f\|_{H^{-1}(\Omega)}+
    \|\nabla (B_hu^h)\|_{L^2(\Omega)},
  \end{align}
  and with the C\'ea lemma
  \begin{align*}
    \|\nabla(y^h-y_h)\|_{L^2(\Omega)} &\le 
    \|\nabla(y^h-B_hu^h)\|_{L^2(\Omega)} = 
    \|\nabla y_0^h\|_{L^2(\Omega)} \\ &\le c\|f\|_{H^{-1}(\Omega)}+
    \|\nabla (B_hu^h)\|_{L^2(\Omega)}.
  \end{align*}

  The remaining term $\|\nabla (B_hu^h)\|_{L^2(\Omega)}$
  is estimated by using the inverse inequality
  \begin{align*}
    \|\nabla (B_hu^h)\|_{L^2(T)}\le ch^{-1/2} \|u^h\|_{L^2(E)}.
  \end{align*}
  for $E\subset T\cap\Gamma$, $T\in\mathcal{T}_h$, which can
  be proved by standard scaling arguments, to get
  \begin{align}\label{eq:serge3}
    \|\nabla (B_hu^h)\|_{L^2(\Omega)}\le ch^{-1/2} \|u^h\|_{L^2(\Gamma)}.
  \end{align}
  Hence we proved
  \begin{align*}
    \|\nabla (y^h-y_h)\|_{L^2(\Omega)}  &\le
    c\|f\|_{H^{-1}(\Omega)} + ch^{-1/2} \|u^h\|_{L^2(\Gamma)}.
  \end{align*}
  With \eqref{thomas:triangleinequality0207}, \eqref{est:regularizationerrornonconvex},
  \eqref{eq:thomas0207}, the previous inequality, and \eqref{eq:uhstable} we finish the proof.
\end{proof}

With \eqref{eq:serge20/06:2} we can immediately conclude the following result.

\begin{corollary}\label{cor:serge}
  Let $\Omega$ be a non-convex domain and let $y_h\in Y_{*h}$ be the
  solution of \eqref {eq:serge20/06:6repeat}, then the discretization error
  estimate
  \begin{align*}
    \|y-(\Pi_R y_h+\alpha(y) \Sd)\|_{L^2(\Omega)}\le Ch^{1/2}
    \left(h^{1/2}\|f\|_{H^{-1}(\Omega)}+ \|u\|_{L^2(\Gamma)}\right)
  \end{align*}
  holds, reminding that $p_s$ and $\alpha(y)$ are given by \eqref{def:ps} and \eqref{eq:serge20/06:7}, respectively.
\end{corollary}

Hence the positive result is that $\Pi_R y_h+\alpha(y)\Sd$ is a better
approximation of $y$ than $y_h$. The problem is that $\Sd$ and $\Fd$
are used explicitly, and in practice they are not known. A remedy of
this drawback is the aim of the next section.

\section{\label{sec:scm}Approximate singular functions}

Following \cite{ciarletjr:03}, we approximate $p_s$ from \eqref{def:ps} by 
\begin{align}\label{eq:psh}
  \begin{split}
    p_s^h &= p_h^* - r_h + r^{-\lambda}\sin(\lambda\theta), \quad
    r_h = B_h \left( r^{-\lambda}\sin(\lambda\theta)\right), \\
    p_h^*&\in Y_{0h}:\quad (\nabla p_h^*,\nabla v_h)_\Omega =
    (\nabla r_h,\nabla v_h)_\Omega \quad\forall v_h\in Y_{0h},
  \end{split}
\end{align}
with $B_h$ from \eqref{def:Rh}. 
The function $\Fd$ from \eqref{def:phis} admits the splitting
\begin{align}\label{eq:serge20/06:11}
  \Fd=\tilde \phi+\beta r^\lambda \sin(\lambda \theta),
\end{align}
with $\tilde \phi\in H^2(\Omega)$ and
$\beta=\pi^{-1}\|\Sd\|^2_{L^2(\Omega)}$, see again
\cite{ciarletjr:03}.  It is approximated by
\begin{align}\label{def:phish}
  \begin{split}
    \phi_s^h &= \phi_h^*-\beta_hs_h+\beta_h r^\lambda\sin(\lambda\theta), \quad
    s_h = B_h \left( r^\lambda\sin(\lambda\theta)\right), \quad
    \beta_h=\frac1\pi\|\Sd^h\|^2_{L^2(\Omega)}, \\
    \phi_h^*&\in Y_{0h}:\quad (\nabla \phi_h^*,\nabla v_h)_\Omega =
    (p_s^h,v_h)_\Omega + 
    \beta_h (\nabla s_h,\nabla v_h)_\Omega \quad\forall v_h\in Y_{0h},
  \end{split}
\end{align}
that means, $\tilde\phi$ is approximated by
$\tilde\phi_h=\phi_h^*-\beta_hs_h\in Y_h$.
The approximation errors are bounded by
\begin{align}
  \label{eq:serge20/06:10}
  \|\Sd-\Sd^h\|_{L^2(\Omega)}&\le ch^{2\lambda-\epsilon}\le ch, \\
  \label{eq:serge20/06:12}
  |\beta-\beta_h|&\le ch^{2\lambda-\varepsilon}\le ch, \\
  \label{eq:serge20/06:13}
  \|\Fd-\Fd^h\|_{1,\Omega}&\le ch,
\end{align}
see \cite[Lemmas 3.1--3.3]{ciarletjr:03}, where
\eqref{eq:serge20/06:12} and \eqref{eq:serge20/06:13} imply
\begin{align}\label{eq:serge20/06:13b}
\|\tilde \phi -\tilde \phi_h\|_{1,\Omega}\le ch.
\end{align}

At the end of Section \ref{sec:3} we saw that $\Pi_R y_h+\alpha(y)\Sd$
is a better approximation of $y$ than $y_h$. Since this function is
not computable we approximate it by 
\begin{align}\label{eq:serge20/06:14}
  z_h=\Pi_R^h\, y_h+\alpha_h \Sd^h,
\end{align}
with
\begin{align}\label{def:PiRhyh}
  \Pi_R^h\, y_h=y_h-\gamma_h
  \Sd^h, \quad \gamma_h=\frac{(y_h, \Sd^h)_\Omega}{\|\Sd^h\|^2_{L^2(\Omega)}}
\end{align}
and a suitable approximation $\alpha_h$ of 
\[
  \alpha(y) =\frac{-(u,\partial_n \Fd)_\Gamma+(f,\Fd)_\Omega}%
  {(\Sd, \Sd)_\Omega}
\]
from \eqref{eq:serge20/06:7}.  To this end we write the problematic
term by using \eqref{eq:serge20/06:11} as
\[
  (u,\partial_n \Fd)_\Gamma=(u,\partial_n \tilde \phi)_\Gamma+
\beta (u,\partial_n (r^\lambda \sin(\lambda \theta)))_\Gamma.
\]
and replace the term $(u,\partial_n \tilde \phi)_\Gamma$ by
$(u^h,\partial_n \tilde \phi)_\Gamma$. Since $\tilde \phi$
belongs to $H^2(\Omega)$ and $u^h$ is the trace of $B_hu^h$, we get by
using the Green formula
\begin{align}\nonumber
  (u^h,\partial_n \tilde \phi)_\Gamma&=
  (B_h u^h,\Delta \tilde \phi)_\Omega +
  (\nabla  B_h u^h,\nabla \tilde \phi)_\Omega \\ \label{eq:thomas0307a} &=
  - (B_h u^h,\Sd)_\Omega +
  (\nabla  B_h u^h, \nabla \tilde \phi)_\Omega
\end{align}
as $\Delta \tilde \phi=\Delta \Fd=-\Sd$.
With all these notations and results, we define
\begin{align}\label{eq:serge20/06:15}
  \alpha_h = \frac{
    (B_h u^h,\Sd^h)_\Omega -
    (\nabla  B_h u^h, \nabla \tilde \phi_h)_\Omega
    - \beta_h(u,\partial_n (r^\lambda \sin(\lambda \theta)))_\Gamma
    + (f,\Fd^h)_\Omega
  }{(\Sd^h,\Sd^h)_\Omega^2}.
\end{align}
Note that $\alpha_h$ can be computed explicitly and therefore $z_h$ as well.

Let us estimate the approximation errors made.
\begin{lemma}\label{lem:4.1}
  Let $\Omega$ be a non-convex domain and let $y_h\in Y_{*h}$ be the
  solution of \eqref {eq:serge20/06:6repeat}. Then the error estimates
  \begin{align}\label{eq:serge20/06:17a}
    \|\Pi_R y_h-\Pi_R^h\, y_h\|_{L^2(\Omega)} &\le
    ch  \left(\|f\|_{H^{-1}(\Omega)}+ \|u\|_{L^2(\Gamma)}\right), \\ 
    \label{eq:serge20/06:20}
    |\alpha(y) -\alpha_h| &\le 
    ch^{1/2}  \left(h^{1/2}\|f\|_{H^{-1}(\Omega)}+ \|u\|_{L^2(\Gamma)}\right)
  \end{align}
  hold.
\end{lemma}

\begin{proof}
  With the definitions of $\Pi_R$ and $\Pi_R^h$, with
  $\gamma:=(y_h,\Sd)_\Omega/\|\Sd\|_{L^2(\Omega)}^2$, and by
  using the triangle inequality we have
  \begin{align*}
    \|\Pi_R y_h-\Pi_R^h\, y_h\|_{L^2(\Omega)} &=
    \|\gamma p_s-\gamma_h p_s^h\|_{L^2(\Omega)} \\ &\le
    |\gamma-\gamma_h|\,\|p_s^h\|_{L^2(\Omega)} + 
    |\gamma|\,\|p_s-p_s^h\|_{L^2(\Omega)}
  \end{align*}
  We write
  \begin{align*}
    \gamma-\gamma_h &= \frac{(y_h, \Sd)_\Omega}{\|\Sd\|^2_{L^2(\Omega)} } -
    \frac{(y_h, \Sd^h)_\Omega}{\|\Sd^h\|^2_{L^2(\Omega)} } \\
    &=\frac{(y_h, \Sd-\Sd^h)_\Omega}{\|\Sd\|^2_{L^2(\Omega)} }+ (y_h, \Sd^h)_\Omega
    \left(\frac{1}{\|\Sd\|^2_{L^2(\Omega)}
    }-\frac{1}{\|\Sd^h\|^2_{L^2(\Omega)} }\right)\\
    &=\frac{(y_h, \Sd-\Sd^h)_\Omega}{\|\Sd\|^2_{L^2(\Omega)} }+ (y_h, \Sd^h)_\Omega
    \frac{(p_s^h+p_s,p_s^h-p_s)_{\Omega}}{\|\Sd\|^2_{L^2(\Omega)}\|\Sd^h\|^2_{L^2(\Omega)}},
  \end{align*}
  and by the Cauchy-Schwarz inequality and \eqref{eq:serge20/06:10} we get
  \[
    |\gamma-\gamma_h|\le ch \|y_h\|_{L^2(\Omega)}.
  \]
  We have used that $\|p_s\|_{L^2(\Omega)}$ and $\|p_s^h\|_{L^2(\Omega)}$ can be treated as constants due to the definition of $p_s$ and due to \eqref{eq:serge20/06:10}.
  We conclude with $|\gamma|\le
  c \|y_h\|_{L^2(\Omega)}$, and \eqref{eq:serge20/06:10} that
  \begin{align}
    \|\Pi_R y_h-\Pi_R^h\, y_h\|_{L^2(\Omega)} &\le
   ch \|y_h\|_{L^2(\Omega)}.\label{jonny:Pi_R}
  \end{align}
  In view of the finite element error estimate \eqref{jonny:fe_error} and the standard a priori estimate for the very weak solution,
  \begin{equation*}
    \|y\|_{L^2(\Omega)}\leq c\left(\|f\|_{H^{-1}(\Omega)}+\|u\|_{L^2(\Gamma)}\right),
  \end{equation*}
  see Lemma 2.3 of \cite{ApelNicaisePfefferer2014a}, we have
  \[
    \|y_h\|_{L^2(\Omega)}\leq \|y\|_{L^2(\Omega)}+\|y-y_h\|_{L^2(\Omega)}\leq c\left(\|f\|_{H^{-1}(\Omega)}+\|u\|_{L^2(\Gamma)}\right).
  \]
  This estimate together with \eqref{jonny:Pi_R} proves \eqref{eq:serge20/06:17a}.
\pagebreak[3]

  The proof of the estimate \eqref{eq:serge20/06:20} is based on
  writing the problematic term in the definition of $\alpha(y)$
  without approximation as
  \begin{align*}
    (u,\partial_n \Fd)_\Gamma &=
    (u,\partial_n \tilde \phi)_\Gamma+
    \beta (u,\partial_n (r^\lambda \sin(\lambda \theta)))_\Gamma \\ &=
    (u-u^h,\partial_n \tilde \phi)_\Gamma+
    (u^h,\partial_n \tilde \phi)_\Gamma+
    \beta (u,\partial_n (r^\lambda \sin(\lambda \theta)))_\Gamma \\ &=
    (u-u^h,\partial_n \tilde \phi)_\Gamma
      - (B_h u^h,\Sd)_\Omega +
  (\nabla  B_h u^h, \nabla \tilde \phi)_\Omega+ 
    \beta (u,\partial_n (r^\lambda \sin(\lambda \theta)))_\Gamma 
  \end{align*}
  where we used \eqref{eq:thomas0307a} in the last step. Consequently,
  we showed that
  \begin{align*}
    \alpha(y)-\alpha_h& =\frac{1}{\|\Sd\|_{L^2(\Omega)}^2} \Big(
    -(u-u^h,\partial_n \tilde \phi)_\Gamma +
     (B_h u^h,\Sd-\Sd^h)_\Omega  -
    (\nabla  B_h u^h, \nabla(\tilde\phi-\tilde\phi_h))_\Omega \\ &\qquad
     -   (\beta-\beta_h)\, 
    (u,\partial_n (r^\lambda \sin(\lambda \theta)))_\Gamma + 
    (f,\Fd-\Fd^h)_\Omega\Big).
  \end{align*}
  To prove \eqref{eq:serge20/06:20}, in view of
  \eqref{eq:serge20/06:10}, \eqref{eq:serge20/06:12}, and
  \eqref{eq:serge20/06:13} it remains to show that
  \begin{align*}
    \left|(u-u^h,\partial_n \tilde \phi)_\Gamma\right| 
    &\le ch^{1/2}\|u\|_{L^2(\Gamma)},\\
    \left|(B_h u^h, \Sd-\Sd^h)_\Omega\right| 
    &\le ch^{1/2}\|u\|_{L^2(\Gamma)}, \\
    \left|(\nabla B_h u^h,\nabla(\tilde \phi-\tilde \phi^h))_\Omega\right|
    &\le ch^{1/2}\|u\|_{L^2(\Gamma)}.
  \end{align*}
  The first estimate follows from the estimate \eqref{eq:carstensen}
  and the fact that $\tilde \phi$ belongs to $H^2(\Omega)$.  The
  second one follows from the Cauchy-Schwarz inequality and the
  estimates \eqref{eq:serge3} and \eqref{eq:serge20/06:10}. Similarly,
  the third estimate follows from the Cauchy-Schwarz inequality and the
  estimates \eqref{eq:serge3} and \eqref{eq:serge20/06:13b}.
\end{proof}

\begin{corollary}
  Let $\Omega$ be a non-convex domain and let $y_h\in Y_{*h}$ be the
  solution of \eqref {eq:serge20/06:6repeat} and let $z_h$ be derived
  by \eqref{eq:serge20/06:14}, \eqref{def:PiRhyh}, and
  \eqref{eq:serge20/06:15}, then the discretization error estimate
  \begin{align*}
    \|y-z_h\|_{L^2(\Omega)}\le Ch^{1/2}  
    \left(h^{1/2}\|f\|_{H^{-1}(\Omega)}+ \|u\|_{L^2(\Gamma)}\right)
  \end{align*}
  holds.
\end{corollary}

\begin{proof}
  The main ingredients of the proof were already derived. Indeed, it is
  \begin{align*}
    \|y-z_h\|_{L^2(\Omega)}&=
    \|\Pi_Ry+\alpha(y) \Sd-\Pi_R^h\, y_h-\alpha_h \Sd^h \|_{L^2(\Omega)} \\ &\le
    \|\Pi_Ry-\Pi_Ry_h\|_{L^2(\Omega)} +
    \|\Pi_Ry_h-\Pi_R^h\, y_h\|_{L^2(\Omega)} + \\ &\qquad
    |\alpha(y)-\alpha_h| \,\| \Sd \|_{L^2(\Omega)} +
    |\alpha_h|\, \| \Sd- \Sd^h \|_{L^2(\Omega)}.
  \end{align*}
  
  The first three terms can be estimated by using Lemmas \ref{lem:3.1}
  and \ref{lem:4.1}. So it remains to treat the fourth term.
  To bound $|\alpha_h|$ we use the triangle inequality
  \[
    |\alpha_h|\leq |\alpha_h-\alpha(y)|+|\alpha(y)|.
  \]
  For the first term we use \eqref{eq:serge20/06:20}, while for the
  second term we use \eqref{eq:serge20/06:7} reminding that $\phi_s$
  belongs to $H^{3/2+\epsilon}(\Omega)$ with some $\epsilon>0$. 
  Altogether we have 
  \[
    |\alpha_h|\leq C
    \left(\|f\|_{H^{-1}(\Omega)}+ \|u\|_{L^2(\Gamma)}\right)
  \]
  and conclude by using \eqref{eq:serge20/06:10}.
\end{proof}

Before we describe the numerical experiments, let us summarize the
algorithm.
\begin{enumerate}
\item Compute the finite element solution
  \begin{align*}
    y_h\in Y_{*h}:\quad (\nabla y_h,\nabla v_h)_\Omega =
    (f,v_h)_\Omega\quad\forall v_h\in Y_{0h}
  \end{align*}
  where $Y_{*h}=\{v_h\in Y_h: v_h|_\Gamma=u^h\}$, compare
  \eqref{eq:serge20/06:6}, with $u^h\in Y_h^\partial$ being an
  approximation of the boundary datum $u$ satisfying
  \eqref{eq:uhstable} and \eqref{eq:carstensen}. 
\item Compute the approximate singular functions:
  \begin{align*}
    r_h &= B_h \left( r^{-\lambda}\sin(\lambda\theta)\right), \\
    p_h^*&\in Y_{0h}:\quad (\nabla p_h^*,\nabla v_h)_\Omega =
    (\nabla r_h,\nabla v_h)_\Omega \quad\forall v_h\in Y_{0h}, \\
    \tilde p_h&= p_h^* - r_h, \\
    \beta_h&=\frac1\pi
    \|\tilde p_h + r^{-\lambda}\sin(\lambda\theta)\|^2_{L^2(\Omega)}, \\
    s_h &= B_h \left( r^\lambda\sin(\lambda\theta)\right), \\
    \phi_h^*&\in Y_{0h}:\quad (\nabla \phi_h^*,\nabla v_h)_\Omega =
    ( \tilde p_h + r^{-\lambda}\sin(\lambda\theta),v_h)_\Omega +  
    \beta_h (\nabla s_h,\nabla v_h)_\Omega \quad\forall v_h\in Y_{0h}, \\
    \tilde\phi_h &= \phi_h^*-\beta_hs_h, 
  \end{align*}
  compare \eqref{eq:psh} and \eqref{def:phish}.
\item Compute 
  \begin{align*}
    \gamma_h &= \frac{(y_h, \Sd^h)_\Omega}{(\Sd^h,\Sd^h)_\Omega} 
    \quad\text{with }p_s^h=\tilde p_h+r^{-\lambda}\sin(\lambda\theta),\\
    \alpha_h &=\frac{
    (B_h u^h,\Sd^h)_\Omega -
    (\nabla  B_h u^h, \nabla \tilde \phi_h)_\Omega
    - \beta_h(u,\partial_n (r^\lambda \sin(\lambda \theta)))_\Gamma
    + (f,\Fd^h)_\Omega
  }{(\Sd^h,\Sd^h)_\Omega^2}, \\
    \delta_h&=\alpha_h-\gamma_h, \\
    \tilde  z_h&=y_h+\delta_h\tilde p_h,
  \end{align*}
  compare \eqref{def:PiRhyh} and \eqref{eq:serge20/06:15}. According
  to \eqref{eq:serge20/06:14}, the numerical solution is
  \begin{align*}
    z_h&=\tilde z_h + \delta_hr^{-\lambda}\sin(\lambda\theta).
  \end{align*}
\end{enumerate}\pagebreak[3]
  Note that all integrals with $r^\lambda$ and $r^{-\lambda}$ must be
  computed with care.

\section{Numerical experiments}

This section is devoted to the numerical verification of our
theoretical results.  For that purpose we present examples with known
solution. Furthermore, to examine the influence of the corner
singularities, we consider several polygonal domain $\Omega_\omega$
depending on an interior angle $\omega\in(0,2\pi)$. The computational
domains are defined by
\begin{equation}\label{eq:compdomain}
  \Omega_\omega:=(-1,1)^2\cap
  \{x\in \R^2: (r(x),\theta(x))\in(0,\sqrt{2}]\times[0,\omega]\},
\end{equation}
where $r$ and $\theta$ stand for the polar coordinates located at the
origin. The boundary of $\Omega_\omega$ is denoted by $\Gamma_\omega$%
.  
We solve the problem
\begin{align}\label{eq:linexample}
    -\Delta y   &= 0 \quad \text{in }\Omega_\omega, &
    y &= u \quad \text{on }\Gamma,
\end{align}
numerically by using the proposed dual singular function method. The
boundary datum $u$ is chosen as follows
\[
  u:=r^{-0.4999}\sin(-0.4999\theta)\quad \text{on } \Gamma_\omega.
\]
This function belongs to $L^p(\Gamma)$ for every $p<2.0004$. The exact
solution of our problem is simply
\[
  y=r^{-0.4999}\sin(-0.4999\theta),
\]
since $y$ is harmonic.

The quasi-uniform finite element meshes for the calculations are
generated by using a newest vertex bisection algorithm. The
discretization errors for different mesh sizes and the corresponding
experimental orders of convergence are given in Table \ref{tab1} for
different interior angles $\omega=270^\circ$ and $\omega=355^\circ$.
We see that the numerical results confirm the expected convergence
rate $1/2$.

\begin{table}\centering
\begin{tabular}{ccc}
\toprule
mesh size $h$  & $\|e_h\|_{L^2(\Omega_\omega)}$ & eoc \\
\midrule
0.25000        & 0.58725                        & \\
0.12500        & 0.42338                        &0.47201 \\
0.06250        & 0.30318                        &0.48177 \\
0.03125        & 0.21606                        &0.48870 \\
0.01562        & 0.15352                        &0.49302 \\
0.00781        & 0.10888                        &0.49572 \\
0.00390        & 0.07712                        &0.49742 \\
\bottomrule
\end{tabular}
\hfill
\begin{tabular}{ccc}
\toprule
mesh size $h$  & $\|e_h\|_{L^2(\Omega_\omega)}$ & eoc \\
\midrule
0.25000        & 1.02069                        & \\
0.12500        & 0.83402                        & 0.29139 \\
0.06250        & 0.58964                        & 0.50025 \\
0.03125        & 0.41696                        & 0.49991 \\
0.01562        & 0.29506                        & 0.49890 \\
0.00781        & 0.20903                        & 0.49725 \\
0.00390        & 0.14836                        & 0.49462 \\
\bottomrule
\end{tabular}
\caption{\label{tab1}Discretization errors $e_h=y-z_h$ for $\omega=3\pi/2$ (left) and $\omega=355\pi/180$ (right)}
\end{table}


We emphasize that the quadrature formula for the numerical integration
of the integral
\[
  (u,\partial_n(r^\lambda\sin(\lambda\theta)))_\Gamma
\]
has to be adapted in order to get a sufficiently good approximation.
Otherwise, the error due to the quadrature formula dominates the
overall error. In our implementation, we chose for the numerical
integration a graded mesh on the boundary ($h_E\sim hr_E^{1-\mu}$ if
the distance $r_E$ of the boundary edge $E$ satisfies $0<r_E<R$ with
$R$ being the radius of the refinement zone and $\mu$ being the
refinement parameter, and $h_T=h^{1/\mu}$ for $r_E=0$) combined with a
one-point Gauss quadrature rule on each element.  Furthermore, the
grading parameter $\mu$ is chosen such that
\[
  \mu\le2\pi/\omega-1,
\]
which seems to be the correct grading to achieve a convergence order
of $1/2$. For the results presented in Table \ref{tab1} we used
$R=0.1$ and $\mu=2\pi/\omega-1$.

%
%
%
%
%

\bibliographystyle{plain}\bibliography{ApPf}
\end{document}